\documentclass[11pt]{article}
\usepackage{amssymb,amsthm,amsmath}

\theoremstyle{plain}
\newtheorem{thm}{Theorem}
\newtheorem{cor}{Corollary}

\theoremstyle{definition}
\newtheorem{rem}{Remark}

\newtheorem{ack}{Acknowledgements}

\def\cprime{$'$}

\newcommand{\C}{{\mathbb C}} 
 
\newcommand{\abs}[1]{{\left| {#1} \right|}}

\renewcommand{\Re}{\operatorname{Re}}

\begin{document}
\author{Johan Andersson\thanks{Uppsala University, \,  Department of Mathematics, \,  P.O. Box 480, \,  SE-751 06 Uppsala \, \, {Email: {\texttt  {johana@math.uu.se}}}}}

\title{{L}avrent\cprime ev's approximation theorem with nonvanishing polynomials and universality of zeta-functions}

\maketitle

\begin{abstract}
 We  prove a variant of the  {L}avrent\cprime ev's approximation theorem that allows us to approximate a continuous function on a compact set $K$ in $\C$ without interior points and with connected complement, with polynomial functions that are nonvanishing on $K$. 

We use this result to obtain a version of the  Voronin universality theorem for compact sets $K$, without interior points and with connected complement where it is sufficient that the  function is continuous on $K$ and the condition that it is  nonvanishing can be removed. 

This implies a special case of a criterion of Bagchi, which in the general case has been proven to be equivalent to the Riemann hypothesis.

\end{abstract}

\begin{thm}
 Let $K$ be a compact set in the complex plane  with connected complement and  without interior points. Let $f(z)$ be any continuous function on $K$. Then for any $\epsilon>0$ there exists a polynomial $p(z)$ that is nonvanishing on $K$ such that
  $$|p(z)-f(z)|<\epsilon,  \qquad \qquad z \in K.$$
\end{thm}
This theorem is  a generalization  of a theorem of  {L}avrent\cprime ev \cite{Lavrentiev} or \cite{Mergelyan2}, which in turn  is a special case of a theorem of Mergelyan where the set $K$ is allowed to have interior points, (see  \cite{Mergelyan}, \cite{Carleson} or \cite{Rudin}, Theorem 20.5).  {L}avrent\cprime ev's original theorem does not say that we can choose the polynomial to be nonvanishing on $K$.

\begin{proof}
 By  {L}avrent\cprime ev's theorem \cite{Lavrentiev} (Theorem 1 without assuming that $p(z)$ is nonvanishing) we can find a polynomial $P(z)$ such that
 \begin{gather} \label{u1} \abs{P(z)-f(z)}<\epsilon/2,  \qquad \qquad z \in K. \\ \intertext{Let} 
 \notag P(z)=c_0 \prod_{k=1}^m (z-z_k), \\ \intertext{where $z_k$ denote the roots of $P(z)$. Since $K$ contains no interior points there exist  sequences $z_{k,n}$ of points in $\C \setminus K$ such that $\lim_{n \to \infty} z_{k,n}=z_k$. Define}
 \notag  p_n(z)=c_0 \prod_{k=1}^m (z-z_{k,n}) \\ \intertext{
Since  all the coefficients of the polynomials $p_n(z)$ will converge to the coefficients of $P(z)$, it is clear that $p_n(z)$ will converge to $P(z)$ uniformly on the compact set $K$. Hence there exists an $n$ such that}
\label{u2} |p_n(z)-P(z)|<\epsilon/2.
\end{gather}
 Since  $z_{k,n}$ denote points in $\C \setminus K$, the polynomial $p_n(z)$ will have all its zeroes outside of $K$ and the polynomial will be nonvanishing on $K$. We can therefore choose $p(z)=p_n(z)$. By the triangle inequality and \eqref{u1}, \eqref{u2} we have  the inequality $|p(z)-f(z)|<\epsilon$ for every $z \in K$.
\end{proof}

\begin{thm}
 Suppose that $K \subset \{z \in \C: 1/2<\Re(z)<1 \}$ is a compact set without interior points and with connected complement, and that $f$ is a continuous function on $K$. Then for any $\epsilon>0$  we have that
$$\liminf_{T \to \infty} \frac 1 T \mathop{\rm meas} \left \{t \in [0,T]:\max_{z \in K} \abs{\zeta(z+it)-f(z)}<\epsilon \right \}>0. $$
\end{thm}

 \begin{proof}
  By Theorem 1 we can approximate $f(z)$ by a polynomial $g(z)$ such that \begin{gather} \label{ui} \abs{g(z)-f(z)}< \epsilon/2,  \qquad \qquad z \in K, \end{gather}
where $g(z)$ is nonvanishing on $K$. By the classical theorem of universality (which in the case of compact sets without interior points is just Theorem 2 where  $f(z)$ is assumed to be nonvanishing on  $K$, see \cite{Bagchi}, \cite{Steuding}, Theorem 1.9  or  \cite{Laurincikas}), we  have that  
$$\liminf_{T \to \infty} \frac 1 T \mathop{\rm meas} \left \{t \in [0,T]:\max_{z \in K} \abs{\zeta(z+it)-g(z)}<\epsilon/2 \right \}=\delta>0. $$
When we have the inequality
$$\max_{z \in K} \abs{\zeta(z+it)-g(z)}<\epsilon/2$$
it follows from  the triangle inequality and \eqref{ui} that 
$$\max_{z \in K} \abs{\zeta(z+it)-f(z)}<\epsilon.$$
Hence 
$$\liminf_{T \to \infty} \frac 1 T \mathop{\rm meas} \left \{t \in [0,T]:\max_{z \in K} \abs{\zeta(z+it)-f(z)}<\epsilon \right \}\geq\delta>0. $$
 \end{proof}
A simple corollary is about universality on lines:
\begin{cor}
 Let $f$ be a continuous function on the interval $[0,C]$ and let $1/2<\sigma<1$. Then for any $\epsilon>0$ there exists a $T$ such that
$$ \max_{t \in [0,C]}\abs{f(t)-\zeta(\sigma+iT+it)}<\epsilon.$$
\end{cor}
 \begin{proof} 
  Choose $K=[\sigma,\sigma+iC]$ in Theorem 2. \end{proof}
Another consequence is that the zeta-function fullfills certain self-similarity properties. More precisely:
\begin{cor} Suppose that $K \subset \{z \in \C: 1/2<\Re(z)<1 \}$ is a compact set without interior points and with connected complement. Then for any $\epsilon>0$ we have that
$$\liminf_{T \to \infty} \frac 1 T \mathop{\rm meas} \left \{t \in [0,T]:\max_{z \in K} \abs{\zeta(z+it)-\zeta(z)}<\epsilon \right \}>0. $$
 \end{cor}

Bagchi \cite{Bagchi}, p. 322 called the property that Corollary 2 is true for all compact sets in the strip $1/2<\Re(s)<1$, {\em strongly recurrent} and proved that it was equivalent to the Riemann hypothesis. We note that we have proved this property whenever the compact set has a connected complement and empty interior. While this of course is not sufficient to prove the Riemann hypothesis, since Bagchi's proof (\cite{Bagchi}, p. 322) uses Rouche's theorem and depends on the existence of interior points, there exist in fact quite a number of compact sets of this type, such as compact sets with positive measure (an example is the Cartesian product of a fat (or Smith-Volterra) Cantor set (\cite{br}, Section 4.1) and an interval), so the result in itself is nontrivial.

\begin{rem}
 Although we have stated the theorem about universality only for the Riemann zeta-function in this paper, it is also true for other zeta-functions as well, such as the Dirichlet L-functions. Theorem 6.12 in Steuding \cite{Steuding}  on the Selberg class for example,  can be stated without assuming that the function is nonvanishing if the compact set in that theorem has no interior elements. Also  in the theorems about joint universality of Dirichlet L-functions (see  \cite{Bagchi}, \cite{Steuding}) the conditions that the functions are nonvanishing can be removed if only compact sets without interior points are considered. 
\end{rem}

\begin{ack}
 The author is grateful to Anders Olofsson for stimulating discussions about this paper.
\end{ack}

\bibliographystyle{plain}

\end{document}